\documentclass[12pt,reqno]{amsart} 

\usepackage[
text={440pt,575pt},
headheight=9pt,
centering
]{geometry}

\usepackage{hyperref}
\usepackage{xcolor}

\definecolor{darkred}{RGB}{160,0,0}
\definecolor{darkblue}{RGB}{0,0,160}
\hypersetup{
  colorlinks,
  citecolor=darkblue,
  filecolor=black,
  linkcolor=darkblue,
  urlcolor=darkblue
}

\usepackage{graphicx}				% Use pdf, png, jpg, or epsÂ§ with pdflatex; use eps in DVI mode
\usepackage{amssymb}
\usepackage{mathtools,amsmath}
\usepackage{amsthm}
\usepackage{amssymb}
\usepackage{mathrsfs}
\usepackage{epstopdf}
\usepackage{color}
\usepackage{enumerate}
\usepackage[capitalise]{cleveref}
\usepackage{tikz-cd}

\usepackage[linesnumbered,ruled]{algorithm2e}
\RequirePackage{amsthm,amsmath,amsfonts,amssymb}
\usepackage[utf8]{inputenc}
\usepackage{chngcntr}
\usepackage{float}
\usepackage{xspace}

\theoremstyle{plain}
\newtheorem{prop}{Proposition}[section]

\newtheorem{cor}[prop]{Corollary}
\newtheorem{lem}[prop]{Lemma}

\theoremstyle{definition}
\newtheorem{dfn}[prop]{Definition}
\newtheorem{rem}[prop]{Remark}
\newtheorem{example}[prop]{Example}

\newcommand{\C}{{\mathbb{C}}}

\newcommand{\A}{{\mathbb{A}}}

\newcommand{\Q}{{\mathbb{Q}}}
\newcommand{\N}{{\mathbb{N}}}
\newcommand{\Z}{{\mathbb{Z}}}

\DeclareMathOperator{\GL}{GL}
\DeclareMathOperator{\Sym}{Sym}
\newcommand{\vc}{\mathcal{V}}

\renewcommand{\subset}{\subseteq}

\newcommand{\cc}{change of coordinates\xspace}
\newcommand{\gb}{Gr\"obner basis\xspace}
\newcommand{\ideal}[1]{\left\langle #1 \right\rangle}

\renewcommand{\setminus}{\smallsetminus}
\renewcommand{\epsilon}{\varepsilon}
\renewcommand{\theta}{\vartheta}

\newcommand{\todfn}[1]{\emph{#1}}

\newcommand\scalemath[2]{\scalebox{#1}{\mbox{\ensuremath{\displaystyle #2}}}}

\allowdisplaybreaks  % Allows displays to break on the page break.

\newcommand{\g}{\mathfrak{g}}
\renewcommand{\c}{\mathfrak{c}}
\renewcommand{\t}{\mathfrak{t}}
\newcommand{\n}{\mathfrak{n}}
\newcommand{\ad}{\mathrm{ad}}

\usepackage{todonotes}
\let\oldnl\nl% Store \nl in \oldnl
\newcommand{\nonl}{\renewcommand{\nl}{\let\nl\oldnl}}% Remove line number for one line

\newcommand{\graphedgestyle}[1]{\tikzset{#1/.style={draw=black}}}
\NewDocumentCommand{\graph}{m >{\SplitList{,}}m O{scale=0.25}}{%
\begin{tikzpicture}[
  anchor = base, baseline,
  Eij/.style = {draw=none},
  Eik/.style = {draw=none},
  Eil/.style = {draw=none},
  Ejk/.style = {draw=none},
  Ejl/.style = {draw=none},
  Ekl/.style = {draw=none},
  #3
]

\tikzset{every node/.style={draw,shape=circle,fill=black,minimum size=2pt,inner sep=0}}
\tikzset{every edge/.style={line width=1pt}}
\ProcessList{#2}{\graphedgestyle}

\ifthenelse{#1 > 1}{
\node (Ni) at (0, 1) {};
\node (Nj) at (1, 1) {};

\path (Ni) edge [Eij] (Nj);
}{}

\ifthenelse{#1 > 2}{
\node (Nk) at (1, 0) {};

\path (Ni) edge [Eik] (Nk);
\path (Nj) edge [Ejk] (Nk);
}{}

\ifthenelse{#1 > 3}{
\node (Nl) at (0, 0) {};

\path (Ni) edge [Eil] (Nl);
\path (Nj) edge [Ejl] (Nl);
\path (Nk) edge [Ekl] (Nl);
}{}

\end{tikzpicture}
}

\title[Deciding toricness effectively] {Efficiently deciding if an
  ideal is toric after a linear coordinate change}

\author{Thomas Kahle \and Julian Vill}
\date{\today}

\makeatletter
\@namedef{subjclassname@2020}{\textup{2020} Math. Subj. Classification}
\makeatother

\begin{document}

\begin{abstract}
  We propose an effective algorithm that decides if a prime ideal in a
  polynomial ring over the complex numbers can be transformed into a
  toric ideal by a linear automorphism of the ambient space.  If this
  is the case, the algorithm computes such a transformation
  explicitly.  The algorithm can compute that all Gaussian graphical
  models on five vertices that are not initially toric cannot be made
  toric by any linear~\cc. The same holds for all Gaussian conditional
  independence ideals of undirected graphs on six vertices.
\end{abstract}

\maketitle

\section{Introduction}

Let $k$ be a field.  A \emph{binomial} in the polynomial ring
$k[x_{1},\dotsc, x_{n}]$ is an element of the form
$x^{a} - \lambda x^{b}$, where $\lambda \in k$, $a,b\in\Z_+^n$.  A
binomial is \emph{unital} if $\lambda = 0$ or $\lambda = 1$ and hence
monomials are also (unital) binomials.  An ideal
$I\subset k[x_{1},\dotsc, x_{n}]$ is \emph{toric} if it is both prime
and can be generated by unital binomials.
Toric ideals are a mainstay of combinatorial commutative algebra
because they present affine semigroup
algebras~\cite[Chapter~10]{MillerSturmfels}.  They are also the
building blocks of toric geometry from which they have their
name~\cite[Prop.~1.1.9]{CoxLittleSchenk}.  An affine toric variety is
the Zariski-closure of a torus orbit of a point.  The coordinate rings
of such varieties can be presented by toric ideals.  Our aim here is
to detect the torus structure using methods from Lie theory.

Toric ideals and varieties frequently appear in applications.  For
example the fundamental theorem of Markov bases is a cornerstone of
algebraic statistics and connects generators of toric ideals to Markov
chains used in hypothesis testing~\cite[Theorem~3.1]{diaconissturmfels98}.
Algebraic statistics also provides ideals which are toric, but not
naturally given in the corresponding coordinates.  One can sometimes
find special coordinate changes related to the discrete Fourier transform
that reveal toric structure~\cite{gmn2022,sturmfels2019brownian}.
Conversely, it can also be useful and hard to clarify that a given
rational variety does \emph{not} admit a toric structure.  The first
example of a Bayesian network with this property appears
in~\cite{nicklasson2023}.

In total, it is desirable to test efficiently if a given ideal can be
made a toric ideal by applying a linear coordinate change on the
ambient space or equivalently on the variables $x_{1},\dotsc, x_{n}$
in the polynomial ring.  Of course one would also like to see if, more
generally, any automorphism of ambient affine space makes a variety
toric, but automorphisms of affine space are not well understood.
Hence in this paper, ``\cc'' always means such a linear \cc.

While toric ideal theory is essentially independent of~$k$, we
restrict to $k=\C$ here because the Lie theory machinery that we
employ also has this requirement.  Actual computations can almost
always be done over $\Q$ or number fields, extending the field step by
step whenever this is required for some computation to have a result.

General algorithms to decide if an ideal can be made toric after a \cc
have been given in~\cite{kmm19}.  It shows that this and various
related problems are all amenable to the idea to compute the
(constructible) locus in $\GL_{n}(\C)$, where some property of
interest holds, for example the property of being cut out by
binomials.  The generality of this approach comes at a price in
computational complexity.  To test for toricness after a \cc with this
method requires the computation of comprehensive Gröbner bases and is
hence infeasible in all but the simplest examples.

Recently, a simple criterion to prove that a variety is not toric
after any \cc appeared in~\cite{mp2023}.  The new idea is to consider
the Lie algebra of the subgroup of $\GL_n(\C)$ that fixes the ideal of
the variety (under the natural action that is defined precisely in
Section~\ref{sec:algo}).  We build on this idea and complete it to an
algorithmic test.  In some situations, this group can be determined
explicitly.  The recent preprint \cite{ghl2024} does this for special
varieties including secant varieties.  These methods, however, require
that the variety under consideration is well-understood.

Our approach to the problem is as follows.  If a variety is toric,
there exists a torus of the same dimension acting on the variety with
a dense orbit.  The criterion of~\cite{mp2023} exploits this to
conclude that, if the locus in $\GL_n(\C)$ which acts on the variety
$\vc(I)$ has dimension strictly smaller than $\vc(I)$, there cannot
exist such a torus.  But the subgroup acting on $\vc(I)$ is a Lie
group itself.  Therefore we search for maximal tori in that Lie group.
This can be done in the corresponding Lie algebra and thus eventually
reduces to linear algebra.  Specifically, we try to find a subalgebra
that is simultaneously diagonalizable.  A coordinate change that
diagonalizes this subalgebra puts $I$ in binomial form if this is
possible at all.  In order to find a simultaneously diagonalizable
subalgebra we first find a Cartan subalgebra which is easy.  Inside
that, finding (as a direct summand) a simultaneously diagonalizable
subalgebra is essentially only Jordan decomposition.

Methods based on Lie algebras of varieties are also applied in 
cryp\-to\-gra\-phy, where proposed cryptographic protocols rely on the 
hardness of finding a linear projective equivalence between two given
projective varieties $X, X'$.  Often one of them is a nice variety
with a large Lie algebra, e.g.~a Veronese embedding of projective space.
This related problem has been 
approached by considering isomorphisms between the two Lie
algebras of $X$ and $X'$ by de Graaf et al.~\cite{deGraaf2006Lie}.
Adjusting these methods to work in finite characteristic 
can yield polynomial time attacks on said protocols~\cite{Castryck25}.

In Section~\ref{sec:lie} we briefly review relevant definitions of Lie
theory.  In Section~\ref{sec:algo} we give our
Algorithm~\ref{alg:decideToric}, discuss its validity and present
several examples from the literature.  We performed some challenging
computations and report on our experiences in \cref{sec:details}.
While the development here is written for homogeneous ideals, a simple
modification extends it to arbitrary ideals.  See
Section~\ref{sec:nonhomog}.  Finally, we present some examples from
algebraic statistics in Section~\ref{sec:gaussGM}.

\section{Lie groups and Lie algebras}
\label{sec:lie}

We briefly review some Lie theory as it pertains to subgroups of
$\GL_{n}(\C)$ that fix given ideals.  Such groups should be kept in
mind in the following.  All facts that are stated without proof here
are to be found in textbooks like \cite{borel1991}.

Let $G\subset\GL_n(\C)$ be any closed linear algebraic group.  Then
$G$ is a smooth subvariety of $\A^{n^2}$.  Let $G^0$ be its identity
component, the unique irreducible component containing the identity
matrix.  This is a normal subgroup of $G$ and $G/G^0$ is finite.
Moreover, all finitely many irreducible components are translates
$gG^0$ for some $g\in G$.  Therefore, $\dim G=\dim G^0$ (as affine
varieties or smooth manifolds).  A torus $T\subset G$ is a subgroup
that is simultaneously diagonalizable and thus has character lattice
$\Z^{k}$ and is homeomorphic to $(\C^\star)^k$ for some $k\in\N$.
In particular, by connectedness (or algebraically irreducibility),
each torus is contained in the identity component~$G^0$.  Hence, in
order to find a maximal torus in $G$ it suffices to consider~$G^0$.
Replacing $G$ with $G^0$ we may assume that $G$ is connected (or
irreducible) when searching for algebraic tori.

Let $M_{n}(\C)$ be the space of complex $n\times n$-matrices and
$\g\subset M_{n}(\C)$ the Lie algebra of~$G$.  This is a finite
dimensional algebra, equipped with the Lie bracket multiplication
$\g\times \g \to \g$ given by the commutator $[A,B] = AB - BA$.  As a
vector space, $\g$ is isomorphic to the tangent space of $G$ at the
identity matrix.  If $T\subset G$ is a torus, the Lie algebra $\t$
of~$T$ is a subalgebra of $\g$ that is simultaneously
diagonalizable. Conversely, any simultaneously diagonalizable
subalgebra of $\g$ is the Lie algebra of a torus in~$G$.  Let $\g$ be
the (complex) Lie algebra of a connected linear algebraic
group~$G$. For any $x\in\g$ the adjoint, $\ad(x) \colon \g\to\g$, is
the endomorphism of $\g$ given by $y \mapsto [x,y]$.

Our eventual goal is to compute a maximal torus inside a group~$G$. If
$G$ is semisimple or reductive, maximal tori of $G$ are in bijection
with Cartan subgroups which we define next.  However, if the group is
not reductive, Cartan subgroups can strictly contain the unique
maximal torus.

\begin{dfn}
\label{dfn:cartan_algebra}
A subalgebra $\c\subset\g$ is a \todfn{Cartan subalgebra} of $\g$ if
$\c$ is nilpotent and self-normalizing, i.e.\ it satisfies the
following two conditions:
\begin{enumerate}
\item there exists $n\in\N$ such that for any $x\in\c$ the map
  $\ad(x)^n$ is zero, and
\item $\c=\{x\in\g\colon [x,\c]\subset \c\}$.
\end{enumerate}
\end{dfn}

\begin{dfn}
  A subgroup $C\subset G$ is a \todfn{Cartan subgroup} of $G$ if it is
  the centralizer of a maximal torus of $G$.
\end{dfn}

% We use the definitions of Cartan groups and algebras from
% \cite{borel1991} or equivalently \cite{bs1968} which do not require
% $G$ to be semisimple.

The reason we consider Cartan groups and Cartan algebras is the
following fact from \cite[Chapter~IV 12.1]{borel1991}, which
eventually allows us to find maximal tori in~$G$.
\begin{prop}
\label{prop:cartan_group_decomposition}
Let $C$ be a Cartan subgroup of $G$, then there is a decomposition
$C=C_s\times C_n$ where $C_s$ is the (closed) subgroup consisting of
all semi-simple elements and $C_n$ is the (closed) subgroup consisting
of all unipotent elements of~$C$. Additionally, $C_s$ is a maximal
torus in~$G$.
\end{prop}

As one would hope, Cartan groups and Cartan algebras are nicely
related in our case of a connected linear algebraic group. For more
general groups, the correspondence of Cartan groups and Cartan
algebras is not true.

\begin{prop}[{\cite[Proposition~6.6]{bs1968}}]
  Cartan subalgebras of $\g$ are exactly the Lie algebras of Cartan
  subgroups of $G$.
\end{prop}

This allows us to also get the decomposition on the Lie algebra side.

\begin{cor}
\label{cor:cartan_algebra_decomposition}
Let $\c\subset\g$ be a Cartan subalgebra. Then there is a
decomposition $\c=\t\oplus\n$ in which $\t$ consists of all
diagonalizable elements of $\c$ and $\n$ consists of all nilpotent
elements of $\c$. Furthermore, the \todfn{toral subalgebra} $\t$ is
simultaneously diagonalizable.
\end{cor}
%\begin{proof}
%If $\c$ is the Lie algebra of the Cartan subgroup $C$ of $G$ the statement follows from the last proposition.
%\end{proof}

In order to compute Cartan subalgebras of $\g$ we use the following
fact.

\begin{prop}[{\cite[Lemma~6.2 and Proposition~6.7]{bs1968}}]
\label{prop:cartan_algebra_kernel}
Let $x\in\g$ be generic. Then $\ker(\ad(x)^{\dim\g})$ is a Cartan
subalgebra of $\g$.  Any Cartan subalgebra has this form.
\end{prop}

\begin{rem}
  By \cite[\S~6]{bs1968}, $\ker(\ad(x)^{\dim\g})$ is a Cartan
  subalgebra if and only if this kernel has the minimal possible
  dimension over all such kernels for $x\in\g$. Since $\ad(x)$ is an endomorphism of the finite dimensional vector space $\g$, it can be represented by a matrix. The kernel having minimal dimension is thus equivalent to the non-vanishing of minors of powers of this matrix which shows that the set of such $x$ is indeed Zariski open.
   However, since we do not
  know the minimal value a priori, the two conditions in
  \cref{dfn:cartan_algebra} lend itself better to computations.
\end{rem}

Maximal tori and Cartan subgroups in $G$ are nicely behaved in the
following way.
\begin{prop}[{\cite[Corollary 11.3.(1)]{borel1991}}]
\label{prop:conjugate}
  All maximal tori in $G$ are conjugate and have the same dimension,
  and the same statement holds for all Cartan subgroups of~$G$.
\end{prop}

\section{The algorithm}
\label{sec:algo}

The group $\GL_n(\C)$ acts on $A:=\C[x_1,\dots,x_n]$ via
$g.f(x)=f(g^{-1}.x)$ ($f\in A$, $g\in\GL_n(\C)$) where $g.x$ is the
usual action of $\GL_n(\C)$ on $\C^n$ via left multiplication.  From
now on, let $I\subset \C[x_1,\dots,x_n]$ be a homogeneous ideal unless
otherwise stated.  It need not be a prime ideal.  We denote by
$G\subset\GL_n(\C)$ the algebraic group that fixes~$I$, i.e.\ the
largest group $G$ with $G.I\subset I$.
Algorithm~\ref{alg:decideToric} takes as input a homogeneous
ideal~$I\subset A$ and tests if it can be made a binomial ideal using
a \cc $S \in \GL_{n}(\C)$.  Any prime ideal that is generated by
binomials can be generated by unital binomials after rescaling of
coordinates (so that the partial character in
\cite[Corollary~2.6]{es1996} becomes the constant~$1$).  Therefore a
binomial prime ideal can be made toric by a linear \cc and hence
Algorithm~\ref{alg:decideToric} also decides if an ideal can be made
toric after a \cc.

We have implemented Algorithm~\ref{alg:decideToric} in
\texttt{SageMath}~\cite{sage}.  The code is available at
\[
  \text{\url{https://github.com/villjulian/isToric}}
\]
Our experiments show that new examples are computable, that are
clearly out of reach of the algorithms from~\cite{kmm19}.

\begin{algorithm}
\caption{Decide toricness}
\label{alg:decideToric}
\KwIn{A homogeneous ideal $I\subset\C[x_1,\dots,x_n]$.}  \KwOut{A
  matrix $S\in\GL_n(\C)$ so that $S.I$ is generated by binomials and
  is prime or $\texttt{False}$ if such an $S$ does not exist (or $I$
  is not prime).}
Compute the Lie algebra $\g$ of the group $G\subset\GL_n(\C)$ fixing $I$, i.e. $G.I\subset I$.\\
Pick $x\in \g$ at random and compute $\c=\ker((\mathrm{ad} (x))^{\dim \g})$.\\
Check if $\c$ is a Cartan subalgebra of $\g$. If not go back to line 2.\\
Decompose $\c=\t\oplus \n$.\\
Compute an $S\in\GL_n(\C)$ that diagonalizes~$\t$.\\
Check if $S.I$ is a binomial ideal.  If not, return \texttt{False}.\\
Check if the binomial ideal $S.I$ is prime.  If not return \texttt{False}.\\
Return $S$.
\end{algorithm}

\begin{proof}[Proof of termination]
  In line 2, if $x$ is picked generic, the resulting algebra is a
  Cartan algebra with probability 1 by
  \cref{prop:cartan_algebra_kernel}.  Therefore no infinite loop can
  arise.  There are no further loops and all remaining parts are
  terminating algorithms.
\end{proof}

\begin{proof}[Proof of correctness]
  If $I$ is not a prime ideal or there is no $S\in\GL_n(\C)$ such that
  $S.I$ can be generated by binomials, then Step~6 or~7 returns
  \texttt{False}.
  Assume $I$ is a prime ideal and some linear transform $I' = S'.I$ is
  generated by binomials.
  This means there is a maximal torus $T'$ in $G$ such that the orbit
  under $T'$ of a generic element of $\vc(I)$ is dense in $\vc(I)$.
  We need to show that the same holds for the torus we compute.  Since
  all maximal tori are conjugate by \cref{prop:conjugate} and the
  torus $T$ we compute is maximal by
  \cref{prop:cartan_group_decomposition}, $T$ and $T'$ are conjugate.
  In particular, the orbit of a generic element of $\vc(I)$ under $T$
  is dense in $\vc(I)$.
  Lastly, if $S$ diagonalizes $\t$, the same holds for the torus $T$
  which now acts diagonally on $\A^n$ and has a dense orbit in
  $\vc(I)$. Thus the ideal can be generated by binomials. Since $I$ is
  prime so is~$S.I$.
\end{proof}

\begin{rem}\label{rem:checkPrime}
  Of course it is a necessary condition that $I$ be prime for a
  positive result.  We check for primeness in the end, because testing
  if a binomial ideal is prime (by \cite[Corollary~2.6]{es1996}) is
  much easier than testing a general ideal (using~\cite{EHV,GTZ}).  In
  our experience the binomiality test is often faster than the
  primality test.  Therefore it makes sense to run it first.
\end{rem}

In the following sections, we comment on the individual steps of the
algorithm.

\subsection{Computing the Lie algebra}

To compute the Lie algebra we follow~\cite{mp2023}.  In that paper $I$
is assumed prime, but this can be avoided as follows: Let $I$ be
generated by polynomials $f_1,\dots,f_s$ of degrees
$d_1\le\dots\le d_s$.  We compute the Lie algebra of the subgroup of
$\GL_n(\C)$ that fixes $I_{d_i}$, the degree $d_{i}$ part of~$I$, for
$i=1,\dots,s$ with the algorithm proposed in
\cite[Theorem~26]{mp2023}. The Lie algebra of the group that fixes all
of $I$ is then the intersection of all these Lie algebras.

The action of $\GL_{n}(\C)$ on $A$ also gives rise to an action of
$M_n(\C)$, the Lie algebra of $\GL_n(\C)$, on $A$ via derivation.  For
$g\in M_n(\C)$ and $f\in A$ let
\[
g\star f(x)=\left.\frac{d}{dt}f(e^{-tg}x)\right\vert_{t=0}.
\]
Then $g\star (fh)=f(g\star h)+h(g\star f)$ for $f,h\in A$.  Moreover,
as described in \cite[Definition~19]{mp2023}, $g$ maps any constant to
zero and any variable $x_i$ to $-\sum_{j=1}^n g_{ij} x_j$.

Most importantly, if $g$ is an $n\times n$ matrix having
indeterminates $g_{ij}$ for $(1\le i,j\le n)$ as entries, the
polynomial $g\star f\in\C[x_1,\dots,x_n,(g_{ij})_{i,j}]$ has degree at
most 1 in the variables $(g_{ij})_{i,j}$. This makes computing the Lie
algebra of the group fixing the ideal a linear algebra problem which
can be solved efficiently.

\subsection{The kernel of  a power of a generic adjoint and Cartan algebras}
An $x$ as in Step~2 can be picked, for example, by taking a random
linear combination of any basis using complex normally distributed
coefficients.  Since \cref{dfn:cartan_algebra} is computational, it
can be used to confirm that $\c$ is Cartan and this is efficient. 
% Checking that $\c$ is indeed Cartan using the definition can be done effectively.
Even if the scalars in a linear combination are randomly chosen from
just $0,\pm 1$ we found a Cartan algebra in most of our experimeents.

\subsection{Decomposing a Cartan algebra}
By \cref{cor:cartan_algebra_decomposition} it suffices to decompose a
basis of $\c$ into diagonalizable and nilpotent parts by Jordan
decomposition.  Indeed, for any element, both the nilpotent and the
diagonalizable summand are also contained in~$\c$.  Moreover, the
resulting set of twice as many elements is still generating.  Hence,
every element is now either contained in $\t$ or~$\n$.
  
\subsection{Diagonalizing the toral subalgebra}
In order to diagonalize the toral subalgebra $\t$ and with it the
torus $T$, we pick a generic element of $\t$ and diagonalize it.  By
the following easy lemma this \cc\ diagonalizes all of~$\t$.
% Of course it can then also be checked that this is indeed the case.
\begin{lem} \label{lem:diagonalizing_torus} Let $V\subset M_n(\C)$ be
  a subspace that is simultaneously diagonalizable.  Let $M\in V$ be
  generic and let $S\in\GL_n(\C)$ be a matrix that diagonalizes $M$,
  i.e.\ $S^{-1}MS$ is diagonal. Then $S^{-1}VS$ consists entirely of
  diagonal matrices.\qed
\end{lem}
% \begin{proof}
%   If $T^{-1}VT$ is diagonal, we see that a generic element has minimal
%   eigenspaces in the following sense: any eigenspace of a generic
%   element is contained in a single eigenspace for every other
%   element. We also see that for genericity it suffices to have the
%   maximal number of distinct values on the diagonal.  Reversing the
%   base change $T$ preserves all observations.  A diagonalizing base
%   change for an element with minimal eigenspaces as above diagonalizes
%   the whole space $V$.
% \end{proof}

If the element $M$ is not generic with respect to~$V$, the base change
$S$ need not diagonalize all of~$V$.  This happens for example when
all eigenspaces of a generic element are 1-dimensional, but the chosen
$M$ has a 2-dimensional eigenspace. Then the matrix $S$ that
diagonalizes all of $V$ has as its columns the bases of all the
1-dimensional eigenspaces.  These columns are unique up to scaling and
so is $S$.  However, a matrix diagonalizing only $M$ leaves an entire
choice of basis of a 2-dimensional eigenspace.

\subsection{Binomial and prime binomial ideals}
The binomiality check in Step~6 of the algorithm is necessary because
without knowledge whether $I$ is prime or not, we cannot argue about
the orbits and stabilizers of the action of~$\t$.  If $I$ is prime,
\cref{prop:KnowPrime} shows that this step can be skipped.

Given an ideal $S.I$, to check whether it is a prime binomial ideal we
first determine if it can be generated by binomials, which is
equivalent to the reduced Gr\"obner basis consisting of binomials
by~\cite[Corollary~1.2]{es1996}.  This check can also be implemented
without Gröbner bases using \cite{conradiKahle15}.  Assuming $S.I$ is
binomial and contains no monomials, we can apply
\cite[Corollary~2.6]{es1996}.  This reduces the check for a prime
ideal to a $\Z$-linear algebra computation on the \todfn{lattice of
  exponents} $L_{I}$ of a binomial ideal $I\subset A$ which consists
of all $u-v\in \Z^{n}$ such that $x^{u} - a x^{v} \in I$ for some
$a\in \C\setminus\{0\}$.  This lattice is \todfn{saturated} if
$L_{I} = \{m\in\Z^n\colon dm\in L_{I} \text{ for some } d\in\Z \}$ or
equivalently, if $\Z^{n}/L_{I}$ is free.
\begin{prop}\label{prop:ES}
  Let $I\subset \C[x_1,\dots,x_n]$ be a binomial ideal that contains
  no monomials.  Then $I$ is prime if and only if the lattice
  $L_I\subset\Z^n$ is saturated.
\end{prop}
If $S.I$ contains some variables these variables are also contained in
its reduced Gr\"obner basis.  In this case we can simply work modulo
these variables.  So, without loss of generality, assume that $S.I$
contains no variables.  Then we compute the colon ideals
$(S.I : x_{i})$ for all $i$.  If one of these colon ideals is not
equal to $S.I$, then $S.I$ is not prime.  If all are equal to $S.I$,
then $S.I$ contains no monomials and by \cref{prop:ES} it suffices to
check if $L_{I}$ is saturated.  This can be done by computing the
Smith normal form of a matrix whose columns are generators of $L_{I}$,
thereby exposing the quotient $\Z^{n}/L_{I}$.

\section{Implementation, examples, heuristics}
\label{sec:details}

We gather several examples from the literature as well as remarks on
our experiences.
  
\begin{rem}\label{rem:probabilistic}
  Algorithm~\ref{alg:decideToric} is probabilistic in a weak sense.
  We do pick random elements to get a Cartan algebra and to
  diagonalize the torus, but we can always immediately certify that
  this random choice was indeed generic.  For Cartan algebras this can
  be done using \cref{dfn:cartan_algebra}.  When diagonalizing the
  torus, we simply check if the chosen base change indeed diagonalizes
  a basis of $\t$ by computing the products $S^{-1}AS$ for every basis
  element $A$.
\end{rem}
\begin{rem}\label{rem:torus_explicit}
  Regardless of whether the ideal ends up being toric or not, after
  taking the exponential of the Lie algebra~$\t$, we get a torus $T$
  acting on the variety~$\vc(I)$. More precisely, $T$ is the group
  generated by $\exp(t_1),\dots,\exp(t_s)$ if $t_1,\dots,t_s$ is a
  basis of~$\t$.
\end{rem}

Whenever the group $G^0$ is not a torus itself, there may be
infinitely many maximal tori contained in it. The algorithm takes one
at random and changes coordinates so that this torus acts diagonally.
In particular, even if there exists a `nice' base change according to
some combinatorial interpretation, the algorithm might not find it.
However, if the group is a torus, the base change may have a
combinatorial interpretation. One example where this happens is the
following and also appeared on
\href{https://mathoverflow.net/questions/93224/proving-that-a-variety-is-not-isomorphic-to-a-toric-variety}{MathOverflow}.
\begin{example}
\label{ex:complete_intersection}
Consider the three quadratic forms
\[
  p_1=e t - r y - q u + w o,\quad p_2=w t - q y - r u + e o,\quad p_3=w e - q r - y u + t o,
\]
which are a complete intersection in $\C[q,w,e,r,t,y,u,o]$.  Let
$I = \ideal{p_{1},p_{2},p_{3}}$.  Its variety $\vc(I)\subset\A^8$ is
$5$-dimensional.  Let $G\subset\GL_8(\C)$ be the subgroup that
fixes~$I$.  We compute the Lie algebra $\g$ of~$G$. It has dimension 5
and has the following basis
\newcommand{\myScaleFactor}{0.84}
\begin{gather*}
\scalemath{\myScaleFactor}{\begin{pmatrix}
1 & 0 & 0 & 0 & 0 & 0 & 0 & 0 \\
0 & 1 & 0 & 0 & 0 & 0 & 0 & 0 \\
0 & 0 & 1 & 0 & 0 & 0 & 0 & 0 \\
0 & 0 & 0 & 1 & 0 & 0 & 0 & 0 \\
0 & 0 & 0 & 0 & 1 & 0 & 0 & 0 \\
0 & 0 & 0 & 0 & 0 & 1 & 0 & 0 \\
0 & 0 & 0 & 0 & 0 & 0 & 1 & 0 \\
0 & 0 & 0 & 0 & 0 & 0 & 0 & 1
\end{pmatrix}},
\scalemath{\myScaleFactor}{\begin{pmatrix}
0 & 1 & 0 & 0 & 0 & 0 & 0 & 0 \\
1 & 0 & 0 & 0 & 0 & 0 & 0 & 0 \\
0 & 0 & 0 & 1 & 0 & 0 & 0 & 0 \\
0 & 0 & 1 & 0 & 0 & 0 & 0 & 0 \\
0 & 0 & 0 & 0 & 0 & 1 & 0 & 0 \\
0 & 0 & 0 & 0 & 1 & 0 & 0 & 0 \\
0 & 0 & 0 & 0 & 0 & 0 & 0 & 1 \\
0 & 0 & 0 & 0 & 0 & 0 & 1 & 0
\end{pmatrix}},
\scalemath{\myScaleFactor}{\begin{pmatrix}
0 & 0 & 1 & 0 & 0 & 0 & 0 & 0 \\
0 & 0 & 0 & 1 & 0 & 0 & 0 & 0 \\
1 & 0 & 0 & 0 & 0 & 0 & 0 & 0 \\
0 & 1 & 0 & 0 & 0 & 0 & 0 & 0 \\
0 & 0 & 0 & 0 & 0 & 0 & 1 & 0 \\
0 & 0 & 0 & 0 & 0 & 0 & 0 & 1 \\
0 & 0 & 0 & 0 & 1 & 0 & 0 & 0 \\
0 & 0 & 0 & 0 & 0 & 1 & 0 & 0
\end{pmatrix}},\\
\scalemath{\myScaleFactor}{\begin{pmatrix}
0 & 0 & 0 & 0 & 1 & 0 & 0 & 0 \\
0 & 0 & 0 & 0 & 0 & 1 & 0 & 0 \\
0 & 0 & 0 & 0 & 0 & 0 & 1 & 0 \\
0 & 0 & 0 & 0 & 0 & 0 & 0 & 1 \\
1 & 0 & 0 & 0 & 0 & 0 & 0 & 0 \\
0 & 1 & 0 & 0 & 0 & 0 & 0 & 0 \\
0 & 0 & 1 & 0 & 0 & 0 & 0 & 0 \\
0 & 0 & 0 & 1 & 0 & 0 & 0 & 0
\end{pmatrix}},
\scalemath{\myScaleFactor}{\begin{pmatrix}
0 & 0 & 0 & 0 & 0 & 0 & 0 & 1 \\
0 & 0 & 0 & 0 & 0 & 0 & 1 & 0 \\
0 & 0 & 0 & 0 & 0 & 1 & 0 & 0 \\
0 & 0 & 0 & 0 & 1 & 0 & 0 & 0 \\
0 & 0 & 0 & 1 & 0 & 0 & 0 & 0 \\
0 & 0 & 1 & 0 & 0 & 0 & 0 & 0 \\
0 & 1 & 0 & 0 & 0 & 0 & 0 & 0 \\
1 & 0 & 0 & 0 & 0 & 0 & 0 & 0
\end{pmatrix}}.
\end{gather*}
This Lie algebra is in fact a Cartan algebra and simultaneously
diagonalizable, i.e.\ $G$ is a torus.  A base change making this ideal
binomial is given by 
\[
\scalemath{\myScaleFactor}{\begin{pmatrix}
1 & 1 & 1 & 1 & 1 & 1 & 1 & 1 \\
-1 & 1 & 1 & -1 & 1 & -1 & -1 & 1 \\
1 & -1 & 1 & -1 & 1 & -1 & 1 & -1 \\
-1 & -1 & 1 & 1 & 1 & 1 & -1 & -1 \\
1 & 1 & 1 & 1 & -1 & -1 & -1 & -1 \\
-1 & 1 & 1 & -1 & -1 & 1 & 1 & -1 \\
1 & -1 & 1 & -1 & -1 & 1 & -1 & 1 \\
-1 & -1 & 1 & 1 & -1 & -1 & 1 & 1
\end{pmatrix}}.
\]
A generic element of $\g$ has simple eigenvalues and this matrix
contains the corresponding eigenvectors in its columns. Thus up 
to reordering of the columns and their scaling this matrix is unique.
\end{example}

\begin{rem}
  In the last example as well as in the following, the Lie algebra
  itself is toral.  Then the coordinate change making the ideal
  binomial is unique up to reordering and scaling.  This does not have
  to be the case in general.  Since all maximal tori are conjugate the
  maximal torus is unique if and only if it is normal in~$G^0$.
  However, even in cases where a maximal torus is not unique, one
  can be lucky as in \cref{ex:many_rat_cc} and find nice rational
  coordinate changes.
  In general, finding rational transformations seems very difficult
  as one needs to find a torus that is diagonalizable over~$\Q$.
\end{rem}

\begin{example}
  Consider the ideal from \cite[Example 3]{ss2005} generated by three
  quadratic forms with 8 terms each.  The generators are given as
  \begin{gather*}
    p_{001} p_{010} - p_{000} p_{011} + p_{001} p_{100} - p_{000} p_{101} - p_{011} p_{110} - p_{101} p_{110} + p_{010} p_{111} + p_{100} p_{111},\\
    p_{001} p_{010} - p_{000} p_{011} + p_{010} p_{100} - p_{011} p_{101} - p_{000} p_{110} - p_{101} p_{110} + p_{001} p_{111} + p_{100} p_{111},\\
    p_{001} p_{100} + p_{010} p_{100} - p_{000} p_{101} - p_{011} p_{101} - p_{000} p_{110} - p_{011} p_{110} + p_{001} p_{111} + p_{010} p_{111}.
  \end{gather*} 
  Using Algorithm~\ref{alg:decideToric} we find the \cc given in
  \cite{ss2005} which has been constructed using the discrete Fourier
  transform.  Moreover, we find that the Lie algebra itself is toral,
  proving that this is indeed the unique (up to scaling and permuting)
  \cc making the ideal binomial.
\end{example}

\begin{example}
\label{ex:many_rat_cc}
Consider the ideal of the Gaussian colored path on three vertices where
each vertex has a distinct color and where the two edges $12$ and $23$
share the same color~\cite[Example 7.7]{cmms2023}:
$\ideal{\sigma_{13}\sigma_{22}-\sigma_{12}\sigma_{23},
  \sigma_{12}\sigma_{13}-\sigma_{11}\sigma_{23}-\sigma_{13}\sigma_{23}+\sigma_{12}\sigma_{33}}$.
In that paper the authors give a rational \cc making the ideal toric.
Running our algorithm multiple times results in different changes of coordinates with
this effect.  However, many of them have rational entries.
\end{example}

In the next example we show that our algorithm falls apart for non-reduced
schemes as the binomial structure of the ideal may not be retrieved from
the Lie algebra.
\begin{example}
\label{ex:trivial_lie_algebra}
Consider the binomial ideal
$I = \ideal {x^4,y^4,x^3y-xy^3} \subset \C[x,y]$.  The Lie algebra of
the group fixing this ideal has dimension 1, containing only multiples
of the identity matrix.  Hence the Lie algebra of any ideal $J$,
obtained from $I$ by a generic linear \cc is trivial and one cannot
gain any insight from it.
\end{example}

A special case of binomial structure is when the ideal can in fact be
generated by unital binomials.  By computing the dimensions of maximal
tori we can at least exclude this case, even if $I$ is not prime.
\begin{rem}\label{rem:notPrimeExclude}
  If an ideal $I$ is unital then $\vc(I)$ acts on itself via
  component-wise \todfn{Hadamard multiplication} (see e.g.\
  \cite[Prop.~4.7]{friedenberg2017minkowski}).  Embedding this action
  diagonally into $\GL_n(\C)$, gives a torus of dimension $\dim\vc(I)$
  acting on~$\vc(I)$.  Therefore, by computing a maximal torus using
  Algorithm~\ref{alg:decideToric}, we can prove that an ideal cannot be
  generated by unital binomials if the dimension of a maximal torus is
  strictly smaller than $\dim\vc(I)$.  However, effectively deciding
  if an arbitrary ideal can be generated by unital binomials seems
  difficult as \cref{ex:trivial_lie_algebra} shows.
\end{rem}

Algorithm~\ref{alg:decideToric} relies on inexact arithmetic in~$\C$.
If the input ideal is defined over $\Q$ or a number field, however,
all computations can be carried out in algebraic (and hence
computable) extensions.  This is the approach we use in our
implementation and due to this, the diagonalization in Step~5 of
Algorithm~\ref{alg:decideToric} can take a long time.  It utilizes the
implementation of the field of algebraic numbers in \texttt{SageMath}.

\begin{rem}\label{rem:LieAlgVSGroebner} Computing Gröbner bases can be
  costly, while both the computation of the Lie algebra and a maximal
  torus do not rely on \gb calculations but rather on linear algebra.
  In practice one will often need a Gröbner basis of the ideal under
  consideration, for example to compute the Krull dimension, but it is
  always good to be careful and avoid such steps if possible.  For
  example, for many Gaussian graphical models on 6 vertices (see
  \cref{sec:gaussGM}) we are not even able to compute a single \gb but
  the computation of a maximal torus terminates within seconds.  See
  \cite{conradiKahle15} for how to decide if a homogeneous ideal is
  binomial without Gröbner bases.
\end{rem}

\begin{rem}
  While examples with a complicated nilpotent part can be crafted by
  hand, in all of our experiments with examples from the literature,
  the whole Cartan algebra has always been simultaneously
  diagonalizable directly.
\end{rem}

\begin{rem}
  Most of the time when we needed to pick a generic element it had
  distinct eigenvalues.  This speeds up computations tremendously.
  Moreover, most of the time it seems to be sufficient to pick random
  linear combinations of basis elements with coefficients
  $0,\pm\frac{1}{2}, \pm 1, \pm 2$.  Then the entries of the generic
  matrix are small rationals and computations finish faster.
\end{rem}

\begin{rem}
  If finding a Cartan algebra or diagonalizing a toral algebra takes
  too long, it can be worthwhile to search for a ``sufficiently
  generic'' element.  Say one of the basis elements $M$ of a Lie
  algebra $\g$ has distinct eigenvalues.  Then it is usually faster to
  compute a Cartan algebra using this element instead of taking a
  random linear combination of all basis elements. This might lead to
  a nicer Cartan algebra and thus to a nicer toral algebra that can be
  diagonalized more efficiently.
\end{rem}

As part of our algorithm, we find a maximal torus $T$ acting
faithfully on a non-degenerate irreducible variety~$\vc(I)$.  In this
situation, $\vc(I)$ is sometimes called a \todfn{$T$-variety of
  complexity $c = \dim \vc(I) - \dim(T)$}.  So a toric variety is a
$T$-variety with $c=0$, while an arbitrary variety is a $T$-variety of
complexity $c = \dim \vc(I)$.  See \cite{altmann2012geometry} for a
survey of the combinatorial and geometric aspects of such actions.  By
computing a maximal torus, our methods also compute this notion of
complexity and we can detect complexity zero.  The essential point is,
that if it is known that $I$ is prime, the check if $S.I$ is binomial
in Step~6 is superfluous.
\begin{prop}\label{prop:KnowPrime}
  Let $I\subset\C[x_1,\dots,x_n]$ be a homogeneous prime ideal. Let
  $T$ be a maximal torus acting on $I$. If $\vc(I)$ is non-degenerate,
  i.e.\ not contained in any hyperplane and $\dim\vc(I)=\dim T$, then
  $I$ is toric after some \cc.
\end{prop}
\begin{proof}
  The torus $T$ acts on $\vc(I)$ and the orbit of a generic element
  has dimension $\dim T$ if and only if the stabilizer of a generic
  point is trivial.  Thus $\vc(I)$ is toric if and only if all generic
  stabilizers are trivial.  

  Assume for a contradiction that there exists a Zariski dense subset
  $U\subset\vc(I)$ such that for any point in $U$ its stabilizer is
  not trivial.
  
  After applying a \cc we can assume that $T$ acts diagonally on $X$
  where $X$ is the image of $\vc(I)$ under that \cc. Let $U'\subset X$
  be the image of $U$ under this \cc.  No point with only non-zero
  coordinates can be fixed by any other element than the identity.
  Thus every point in $U'$ has to have some coordinate equal to zero,
  i.e.
  \[
    U'\subset \bigcup_{i=1}^n  \{x_i=0\}.
  \]
  Taking the Zariski closure we get 
   \[
     X\subset \bigcup_{i=1}^n \{x_i=0\}.
   \]
  since $U'$ is dense in $X$ by assumption.
  Therefore
    \[
    X=\bigcup_{i=1}^n \left( X\cap \{x_i=0\}\right)
  \] 
  and since $X$ is irreducible, the union is equal to one of the
  $(X\cap \{x_{i} = 0\})$, showing that $X$ and thus $\vc(I)$ is
  contained in a hyperplane.  This contradicts the assumption on
  $\vc(I)$ and shows that generic points of $\vc(I)$ have trivial
  stabilizers, so $\vc(I)$ is toric.
\end{proof}

The proposition will be useful mostly if one knows from the start that
a given ideal is prime and does not contain a linear form.  Checking
this computationally is often orders of magnitude slower than
Algorithm~\ref{alg:decideToric} (see Remark~\ref{rem:checkPrime}).  If
one has this information, however, this removes any computations with
algebraic numbers.

\begin{rem}
  If the variety under consideration is degenerate, one can always
  restrict to an affine space containing the variety and compute a
  maximal torus there.
\end{rem}

\section{Non-homogeneous ideals}
\label{sec:nonhomog}
An adapted version of Algorithm~\ref{alg:decideToric} can be applied
to non-homogeneous ideals.  In this case it seems more natural to
allow an affine-linear \cc, but if desired, one may also restrict to
linear coordinate changes.  We only spell out the affine-linear case.

Let $I\subset\C[x_1,\dots,x_n]$ be a not necessarily homogeneous
ideal.  Write $I^h$ for its homogenization with respect to a
variable~$x_0$.  For any homogeneous ideal $J\subset\C[x_0,\dots,x_n]$
let $J^d$ be its dehomogenization, i.e.\ the ideal in
$\C[x_1,\dots,x_n]$ obtained by setting $x_0=1$.  Let $\GL_{n+1}(\C)$
act on $\C[x_0,\dots,x_n]$ as usual and consider the subgroup
\[
H:=\{S\in\GL_{n+1}(\C)\colon  e_0 S^{-1}=e_0\}=\{S\in\GL_{n+1}(\C)\colon  e_0 S=e_0\}
\]
where $e_0=(1,0,\dots,0)$.  Let $H$ act on $\C[x_1,\dots,x_n]$ as
$S.f(x)=(S.f^h(x))^d$ for $S\in H$ and $f\in\C[x_1,\dots,x_n]$.  It
holds that $(S.I)^h=S.I^h$.  Explicitly, any $S\in H$ is of the form
\[
S=
\left(
\begin{array}{c|ccc}
1 & 0 & \cdots & 0\\
\hline
\star &  &  & \\
\vdots &  & \star & \\
\star &  &  & \\
\end{array}
\right).
\]
Assume $S\in H$ is such that $S.I$ is toric.  Since any reduced \gb of
$S.I$ consists of binomials, there exists a term order such that
homogenizing the \gb gives rise to a \gb of $(S.I)^h=S.I^h$. Hence,
$I^h$ is toric.

The other implication is not true.  Even if the homogenization can be
made toric, this need not be true for the ideal $I$.
Dehomogenizing the ideal $I$ in \cref{ex:complete_intersection} with
respect to any variable results in a non-homogeneous ideal in 7
variables that cannot be made toric.  Indeed, any \cc that would do so
lies in $H\subset\GL_8(\C)$. However, there is only one (up to scaling
and reordering of the columns) \cc that makes $I$ toric.  It does not
have the required form.

To decide if a potentially non-homogeneous ideal can be made toric
using an affine-linear \cc, Algorithm~\ref{alg:decideToric} needs only
one modification.  Step~1 is extended so that $\g$ is now the Lie
algebra of $H$ above.

\smallskip\hrule\vspace{2pt}
\noindent {\scriptsize\textbf{1'}}$\,\,$ Compute the Lie algebra
$\tilde{\g}$ of the group $\tilde{G}\subset\GL_{n+1}(\C)$ fixing
$I^h$, i.e.\ $\tilde{G}.I^h\subset I^h$.  Then compute the Lie
subalgebra $\g$ of $\tilde{\g}$ consisting of all matrices $A$ with
$e_0=(1,0,\dots,0)^T\in\C^{n+1}$ as an eigenvector of $A^T$.
\vspace{2pt}\hrule\smallskip

% \begin{algorithm}
% \caption{Decide Toricness for non-homogeneous ideals }
% \label{alg:decideToricNonHom}
% \KwIn{An ideal $I\subset\C[x_1,\dots,x_n]$.}  \KwOut{A
%   matrix $S\in H$ so that $S.I$ (with action given above) is generated by binomials and
%   is prime or $\texttt{False}$ if such an $S$ does not exist (or $I$
%   is not prime).}
% Compute the Lie algebra $\g$ of the group $G\subset\GL_{n+1}(\C)$ fixing $I^h$, i.e. $G.I^h\subset I^h$.\\
% Compute the Lie subalgebra $\tilde{\g}$ of $\g$ consisting of all matrices $A$ with $e_0=(1,0,\dots,0)^T\in\C^{n+1}$ as an eigenvector of $A^T$.\\
% Pick $x\in \tilde{\g}$ at random and compute $\c=\ker((\mathrm{ad} (x))^{\dim \tilde{\g}})$.\\
% Check if $\c$ is a Cartan subalgebra of $\tilde{\g}$. If not go back to line 3.\\
% Decompose $\c=\t\oplus \n$.\\
% Compute an $S\in\GL_{n+1}(\C)$ that diagonalizes~$\t$.\\
% Check if $S.I$ is a binomial ideal.  If not, return \texttt{False}.\\
% Check if the binomial ideal $(S.I)^d$ is prime.  If not return \texttt{False}.\\
% Return $S$.
% \end{algorithm}

This algorithm is correct, because the new $\g$ is indeed a Lie
algebra: the bracket of two elements with (left) eigenvector $e_0$ has
$e_0$ in its (left) kernel.  For any $S\in H$, the first row of $S$
is~$e_0$.  Hence $e_0$ is a (left) eigenvector of any element in~$\t$
and $\t$ is contained not only in $\tilde{\g}$ but also in~$\g$.
Everything else follows exactly as for
Algorithm~\ref{alg:decideToric}.

\section{Application to Gaussian graphical models}
\label{sec:gaussGM}

A Gaussian graphical model consists of covariance matrices
(parametrizing centered multivariate normal distributions) with a
given zero pattern of inverse
covariance~\cite[Chapter~5]{lauritzen96}.  Specifically, to an
undirected graph $G=(V,E)$ with vertices $V=\{1,\dots,p\}$ and edges
$E$ the Gaussian graphical model is
\[
  \mathcal{M}(G)= \lbrace \Sigma\in\Sym_n^+\colon (\Sigma^{-1})_{ij}=0
  \text{ for all } ij\notin E \rbrace
\]
where $\Sym_n^+$ denotes the real symmetric, positive definite
$n\times n$-matrices.  The resulting constraints on $\Sigma$ are
polynomials and thus $\mathcal{M}(G)$ is the intersection of an
algebraic variety with~$\Sym_n^+$.  For a more detailed introduction
see e.g.\ \cite{uhler2017}.

In applications it can be useful to know the vanishing ideal of the
Zariski-closure $\overline{\mathcal{M}(G)}$.  These vanishing ideals
are complicated and somewhat mysterious.  It seems that there exists
no combinatorial rule to derive ideal generators from the graph.  In
practice, they can be computed, for example, by examining the minimal
primes of the ideal generated by entries of $\Sigma^{-1}$.  It is
known that there is a unique minimal prime $\mathfrak{p}$ such that
$\vc(\mathfrak{p})\cap\Sym_n^+=\mathcal{M}(G)$ and this can be
computed by saturating at all principal minors.  An interesting
question related to our question here is, which vanishing ideals are
toric without a \cc.  In~\cite{misra2021gaussian} it is shown that all
vanishing ideals that are generated by linear and quadratic
polynomials are toric.  They arise from graphs that are $1$-clique
sums of complete graphs.  It would be interesting to characterize
which vanishing ideals are toric after a \cc.

We applied Algorithm~\ref{alg:decideToric} to Gaussian graphical
models of small graphs.  In the case $p=4$, this yields computations
in a polynomial ring with 10 variables for the 10 entries of a
symmetric $4\times 4$-matrix.  There are 5 connected and not complete
graphs on $4$ vertices up to isomorphism.  The numerical data for the
vanishing ideals of the corresponding graphical models is as follows.
\begin{center}
\begin{tabular}{c|c|c|c|c}
  graph	& dim model & dim Lie algebra & dim max tori & can be made toric\\
  \hline
  diamond $\graph4{Eij,Eik,Eil,Ejk,Ekl}$ & 9 & 30 	& 6 & no\\ \hline
  paw $\graph4{Eij,Eik,Eil,Ejk}$ & 8 & 52 	& 8 & yes\\ \hline
  cycle $\graph4{Eij,Ejk,Ekl,Eil}$ & 8 & 4 	& 4 & no\\ \hline
  claw $\graph4{Eij,Eik,Eil}$ & 7 & 37 	& 7 & yes\\ \hline
  path $\graph4{Eij,Ejk,Ekl}$ & 7 & 33 	& 7 & yes
\end{tabular}
\end{center}
In each case the Cartan subalgebra was already simultaneously
diagonalizable and the computations finished within few seconds of
single-core computation.

For the cycle it has been shown in \cite{mp2023} that no \cc makes it
toric.  For the diamond graph the dimension criterion of \cite{mp2023}
does not apply, since the Lie algebra is sufficiently large.  However,
the dimension of a maximal torus is too small.  All models that can be
made toric are already toric and do not need a \cc.  We included them
here to show Lie algebras of larger dimension.

For graphs on $p=5$ vertices, the polynomial ring has $15$ variables
and the computation of vanishing ideals of graphical models becomes
challenging, but does finish.  As an example, consider the ``diamond
with an extra edge''$\!\!\!$ \raisebox{-1pt}{
  \begin{tikzpicture}[scale=0.26]
  \tikzset{every node/.style={draw,shape=circle,fill=black,minimum size=2pt,inner sep=0}}
  \tikzset{every edge/.style={line width=1pt}}

    \node (N1) at (0, 1) {};
    \node (N2) at (1, 1) {};
    \node (N3) at (1, 0) {};
    \node (N4) at (0, 0) {};
    \node (N5) at (2, 0) {};  

    \path (N1) edge[draw] (N2);
    \path (N2) edge[draw] (N3);
    \path (N3) edge[draw] (N4);
    \path (N4) edge[draw] (N1);
    \path (N1) edge[draw] (N3);
    \path (N3) edge[draw] (N5);
\end{tikzpicture}}.
%
% \begin{tikzpicture}[thick,scale=0.29]	 
% 	 \node[circle, draw, fill=black!0, inner sep=1pt, minimum width=1pt] (1) at (0,0) {};
%  	 \node[circle, draw, fill=black!0, inner sep=1pt, minimum width=1pt] (2) at (3,-2.5) {};
%  	 \node[circle, draw, fill=black!0, inner sep=1pt, minimum width=1pt] (3) at (-3,-2.5) {};
%  	 \node[circle, draw, fill=black!0, inner sep=1pt, minimum width=1pt] (4) at (0,-5) {};
% 	 \node[circle, draw, fill=black!0, inner sep=1pt, minimum width=1pt] (5) at (6,-2.5) {};
% 	\draw[-]   (1) -- (2) ;
% 	\draw[-]   (2) -- (3) ;
% 	\draw[-]   (3) -- (4) ;
%     \draw[-]   (4) -- (2) ;
%     \draw[-]   (1) -- (3) ;
% 	\draw[-]   (5) -- (2) ;
% \end{tikzpicture}
The vanishing ideal of the Zariski closure of the model
$\mathcal{M}(G)$ is (in some labeling) generated by the six quadrics
and two cubics:\vspace{-1.5ex}
\begin{gather*}
 x_{7} x_{14}  -x_{8} x_{13}, \quad
 x_{6} x_{14}  -x_{8} x_{11}, \quad
 x_{6} x_{13}  -x_{7} x_{11}, \\
x_{1} x_{14}   -x_{4} x_{8}, \quad
x_{1} x_{13}   -x_{4} x_{7}, \quad
x_{1} x_{11}   -x_{4} x_{6}, \\
  x_{3} x_{4} x_{11} - x_{4}^{2} x_{10} + x_{2} x_{4} x_{13} - x_{0} x_{11} x_{13} - x_{2} x_{3} x_{14} + x_{0} x_{10} x_{14},\\
  x_{2} x_{3} x_{8} - x_{3} x_{4} x_{6} - x_{2} x_{4} x_{7}  + x_{1} x_{4} x_{10} - x_{0} x_{8} x_{10} + x_{0} x_{7} x_{11}.
\end{gather*}
We find that the dimension of the model is~11 and the dimension of the
Lie algebra is~56.  We find a Cartan algebra of dimension 8 which is
simultaneously diagonalizable.  Hence, maximal tori have dimension 8
and thus this variety is not toric after any \cc.  The vector space
dimension of the degree 3 part of the polynomial ring is~$680$.  Hence
the computations take place in the Lie algebra of $\GL_{15}$ which has
dimension~$225$.  Nevertheless, these computations finish quickly.

We tested all Gaussian graphical models on $p=5$ vertices and none of
them becomes toric after a \cc if it was not toric already.
For $p=6$ vertices, we did not compute the actual vanishing ideals, as
this is often prohibitive.  Instead we consider the conditional
independence ideal of the graph, which can be written down easily and
one of whose minimal primes is the vanishing ideal.  Again, none of
these conditional independence ideals becomes toric after a \cc if it
was not toric already.
Even in the case of $p=7$ vertices computing Lie algebras is possible.
However, one has to do more work by hand, as it is not always possible
to compute a \gb which simplifies computation of the dimension.  As we
do not expect different outcomes for larger $p$ we did not embark on
this journey.

\bibliographystyle{amsplain}
\bibliography{refs}

\bigskip \medskip

\noindent
\footnotesize {\bf Authors' addresses:}

\noindent Thomas Kahle, OvGU Magdeburg, Germany,
{\tt thomas.kahle@ovgu.de}

\noindent Julian Vill, OvGU Magdeburg, Germany,
{\tt julian.vill@ovgu.de}

\end{document}